\documentclass[a4paper,sts,reqno]{my_imsart}

\usepackage{amsthm,amsmath}
\RequirePackage[colorlinks,citecolor=blue,
urlcolor=blue,breaklinks=true]{hyperref}

\usepackage{amssymb,a4wide}
\usepackage{amsfonts,mathrsfs}
\usepackage{graphicx}
\usepackage{url}
\usepackage{pifont,dsfont}
\usepackage{breakcites,microtype}
\usepackage[round,authoryear]{natbib}

\usepackage{cleveref}
\usepackage{autonum}
\usepackage{nicefrac}
\usepackage{bbm, bm}
\usepackage{mathrsfs}

\startlocaldefs
\def\bX{\boldsymbol X}
\def\bY{\boldsymbol Y}
\def\bZ{\boldsymbol Z}
\def\bxi{\boldsymbol\xi}
\def\balpha{{\boldsymbol\alpha}}

\def\bx{\boldsymbol x}

\def\bz{\boldsymbol z}
\def\bw{\boldsymbol w}
\def\bu{\boldsymbol u}

\def\hat{\widehat}
\newtheorem{assLip}{Assumption}

\newtheorem{assBound}{Assumption}

\newtheorem{assGaussian}{Assumption}

\newtheorem{theorem}{Theorem}
\newtheorem{proposition}[theorem]{Proposition}
\newtheorem{corollary}[theorem]{Corollary}
\newtheorem{lemma}[theorem]{Lemma}

\theoremstyle{remark}
\newtheorem{remark}[theorem]{Remark}

\endlocaldefs

\begin{document}

\begin{frontmatter}


\title{Risk bounds for aggregated 
shallow neural networks using Gaussian prior}
\runtitle{Risk bounds for aggregated neural nets}

\begin{aug}
\author{\fnms{Laura} \snm{Tinsi\ead[label=e1]{Laura.Tinsi@ensae.fr}}}
\and
\author{\fnms{Arnak S.} \snm{Dalalyan}\ead[label=e2]{arnak.dalalyan@ensae.fr}}

\runauthor{Tinsi and Dalalyan}

\affiliation{CREST, ENSAE Paris, Institut Polytechnique de Paris}
\address{5 avenue H. le Chatelier, Palaiseau, France}
\end{aug}

\begin{abstract}
Analysing statistical properties of neural networks 
is a central topic in statistics and machine learning. 
However, most results in the literature focus on the
properties of the neural network minimizing the training 
error. The goal of this paper is to consider aggregated
neural networks using a Gaussian prior. The departure
point of our approach is an arbitrary aggregate satisfying
the PAC-Bayesian inequality. The main contribution 
is a precise nonasymptotic assessment of the estimation
error appearing in the PAC-Bayes bound. We also review 
available bounds on the error of approximating a function
by a neural network. Combining bounds on estimation and 
approximation errors, we establish risk bounds which are 
sharp enough to lead to minimax rates of estimation over
Sobolev smoothness classes.
\end{abstract}



\end{frontmatter}

\maketitle

\section{Introduction}

Neural networks are the most widely used parameterised
functions for solving machine learning tasks. The parameters
of the neural network are then learned from data.
Assessing the error of the learned network
on new, unobserved examples is a central topic in 
statistics and learning theory \citep{bartlett_montanari_rakhlin_2021,Fan2021}. The most popular approach for estimating the parameters of
the network from data, referred to as weights and biases, 
 is the minimization of the (regularized) 
training error. 
This is usually done by the stochastic gradient 
descent, or a version of it. Risk bounds for these
networks are based on Vapnik-Chervonenkis dimension 
\citep{BartlettMM98,anthony_bartlett_1999,BartlettHarvey}. Even for simple networks containing only one 
hidden layer, these risk bounds are rather involved 
\citep{Xie17,ZhongS0BD17,CaoG19,Ba2020}. 

A well-known alternative to minimizing the regularized
training error is to use a prediction rule based on a 
posterior distribution. Typical example is the 
network obtained by sampling its weights from the
posterior, or the convex combination of the networks 
averaged using the posterior distribution. 
Surprisingly, little is known about risk bounds 
of posterior-based prediction rules in the context
of neural networks. The goal of the present work is
to do the first step in filling this gap by focusing
on one-hidden-layer feedforward neural networks and 
Gibbs posteriors using Gaussian prior. An attractive
feature of posterior-based methods is that their
analysis can be carried out using the PAC-Bayes theory
\citep{mcallester1999some, mcallester2003pac} as a 
substitute to the Vapnik-Chervonenkis dimension or
the Rademacher complexity. We refer the reader to \citep{catoni2007pac,guedj2019primer,alquier2021userfriendly} for a comprehensive account of the PAC-Bayesian 
approach in statistics and learning. 

PAC-Bayes theory has been already used in the framework 
of neural networks, mainly for providing data-driven 
bounds on the generalisation error of trained (stochastic)
networks and prior selection based on these bounds 
\citep{Rivasplata2018PAC,Lever2013tighter,dziugaite2017computing,neyshabur2017exploring,ZhouVAAO19,LetarteGGL19,BiggsG21,perezortiz2021learning,perezortiz2021progress}. In this paper, we take a different route and 
propose to use in-expectation PAC-Bayes bounds for 
investigating the risk (or, the expected excess loss) of
aggregated neural networks. To be more specific, let
$\mathcal{F}_{\mathsf W}:= \{f_{\bw}, {\bw} \in {\mathsf W}
\}$ be a parametric class of prediction rules, with a parameter $\bw$ lying in a measurable space $({\mathsf W}, \mathscr{W})$. One can think of $\mathcal{F}_{\mathsf W}$ as
a set of neural networks with a given architecture and of
$\bw$ as the vector of the weights and biases. Assume we are
given a sample of size $n$ independently drawn from an 
unknown distribution $\bf P$, and we wish to ``aggregate''
elements of $\mathcal{F}_{\mathsf W}$ to obtain a 
prediction rule $\hat f_n$ that mimics the Bayes predictor
$f_{\mathbf P}$. This means that for a prescribed loss 
function $\ell(\cdot,\cdot)$ taking real values, we 
wish $\ell(\hat f_n, f_{\mathbf P})$ to be small. It turns out
that under some general assumptions, for a given prior 
distribution $\pi$ on $\mathsf W$ and a temperature 
parameter $\beta>0$, there exists an aggregate  
$\hat f_n$ such that  
\begin{align}\label{eq:intro}
    \mathbf E[\ell(\hat f_n,f_{\mathbf P})] \le  C_{\textsf{PB}}
    \inf_{p} \bigg\{ \int_{\mathsf W} 
    \ell(f_{\bw},f_{\mathbf P})\,p(d\bw) + \frac{\beta}{n} \,D_{\textsf{KL}}(p||\pi)\bigg\},
\end{align}
where $C_{\textsf{PB}}$ is some constant and 
the infimum is over all probability distributions $p$ 
over $\mathsf W$. We say then that $\hat f_n$ satisfies 
a PAC-Bayes inequality in-expectation. For regression
with fixed design, the Gibbs-posterior mean was shown
to satisfy \eqref{eq:intro} with $ C_{\textsf{PB}} = 1$ 
in \citep{Leung06} for Gaussian noise, and in \citep{dalalyan2007aggregation} for more general noise
distributions. In some other problems, including the random design regression and the density estimation, similar bounds
were established for the mirror averaging \citep{yuditskii2005recursive,learning2008Juditsky}. 
PAC-Bayes bounds with $C_{\textsf{PB}}>1$ for the 
prediction rule obtained by randomly drawing $\bw$ 
from the Gibbs posterior were proved in \citep{catoni2007pac,AlquierB13}. The recent papers
\citep{BiggsG21,Dubois} studied the problem of aggregation 
of neural networks with sign activation.

In the present work, we elaborate on \eqref{eq:intro} 
to get a tractable risk bound when $\mathcal{F}_{\mathsf W}$ is the set of neural networks with a single hidden layer.  
The tractability here should be understood as the property
of showing clearly the dependence on the important problem
characteristics (sample-size, input and output dimensions) 
and those of the learning algorithm (variances of the prior distribution, number of hidden layers, properties of the 
activation functions). Our first main contribution stated 
in \Cref{prop:rho} is a tractable risk bound formulated as 
an oracle inequality. To our knowledge, this inequality is
sharper and easier to deal with than its counterparts for 
the training error minimizing shallow networks. To show
potential implications of this oracle inequality, we combine
it with known approximation bounds when the Bayes predictor 
lies in a Sobolev ball. Interestingly, we show that a proper
choice of the width of the hidden layer and the variances of
the prior leads to minimax optimal rates of convergence, up
to logarithmic factors. More specifically, for the Sobolev 
ball $W^r_2([0,1]^{D_0})$ of smoothness $r$ and input dimension 
$D_0$, we obtain in \Cref{prop:riskbound} the rate 
$n^{-{2 r}/{(2r+ D_0)}}\log^2 n$ for a specific class of 
sigmoid activation functions. A similar result is obtained for
the ReLU activation as well, but with a slightly slower 
rate $n^{-{2\bar r}/{(2\bar r+ D_0 +1)}}$ (up to a 
polylogarithmic factor) for any $\bar r<r$.

The rest of the paper is organized as follows. In
\Cref{section:prelim}, we define the generic PAC-Bayesian
framework and instantiate it in the setting of shallow 
neural networks. In \Cref{section:oracleBound}, we state 
the main oracle bound for shallow neural networks with a
Gaussian prior.  \Cref{section:examples} provides 
examples of statistical problems where PAC-Bayesian 
bounds of type \eqref{eq:intro}  are available. 
\Cref{section:approxBound} is devoted to a selective 
review of the literature on approximation 
properties of neural networks with bounded (sigmoid) 
and unbounded (ReLU) activation functions. 
Finally,  \Cref{section:riskbound} contains the
upper bounds on the worst-case risk which are nearly 
minimax rate-optimal in the case of sigmoid activation. Some concluding remarks are 
provided in \Cref{section:conclusion}. Technical 
proofs are deferred to the appendices.

\section{Preliminaries and notation}\label{section:prelim}
In this section, we set the general framework of the 
PAC-Bayesian bound that will be the starting point of 
our work. We then instantiate it in the specific case 
of neural networks.

\subsection{General framework and PAC-Bayesian type bounds}\label{section:generalframeworks}

Let $(\mathcal{Z}, \mathscr A)$ be a measurable space. 
We observe one realisation of the random vector
$\mathbf{Z}^n = (Z_1, \dots, Z_n)  \in \mathcal{Z}^n$  
drawn from an unknown distribution $\mathbf{P}$ on 
$(\mathcal{Z}^n, \mathscr A^{\otimes n})$. We denote
by $\|\bx\|_2$ the Euclidean norm of the vector $\bx$
of an Euclidean space. 
Let $\mathcal{X} \subset \mathbb{R}^{D_0}$, 
$D_0 \ge 1$, be a Borel set and  let $\mu$ be a 
$\sigma$-finite measure on $\big(\mathcal{X}, 
\mathscr B(\mathcal{X})\big)$ such that $M_2^2 
= D_0^{-1} \int_{\mathcal X} \|\bx\|_2^2\,\mu(d\bx)
<+\infty$. In the sequel, we denote by $\mathbb 
L_q(\mu)$, $q\in[1,\infty)$, the set of all the 
functions $f:\mathcal X\to \mathbb{R}^{D_2}$ such 
that $\int_{\mathcal X} \|f(\bx)\|_2^q\,\mu
(d\bx)<\infty$. Let $\mathcal{P}_{\mathsf W}$ be 
the space of all probability measures on 
${\mathsf W}$ and let 
\begin{align}
    \mathcal{P}_1(\mathcal{F}_{\mathsf W}) = \Big\{
    p\in \mathcal{P}_{\mathsf W} : \int_{\mathsf W}
    \|f_{\bw}(\bx)\|_2\,p(d{\bw})< \infty,\quad
    \text{for all}\quad x\in\mathcal X\Big\}.
\end{align}

We consider the problem of estimating a function 
$f_{\mathbf P}\in\mathbb L_2(\mu)$. At this stage, one
may think of $f_{\mathbf P}$ as the multidimensional
regression function when $\mathcal Z = \mathcal X 
\times \mathbb R^{D_2}$, the Bayes classifier
when  $\mathcal Z = \mathcal X\times \{-1,1\}$ 
or the density of observations when $\mathcal Z = 
\mathcal X$ (in the last two cases $D_2=1$). 
A common approach in statistics and statistical
learning is to use a parametric set  
$\mathcal{F}_{\mathsf W}:= \{f_{\bw}, {\bw} \in 
{\mathsf W}\} \subset \mathbb L_2(\mu)$, indexed 
by a measurable set $\mathsf W\subset \mathbb R^d$, 
for some $d\in\mathbb N$, for constructing an 
estimator of $f_{\mathbf P}$. Instances of this 
approach are the empirical risk minimizer, 
the Bayesian posterior mean, the exponentially 
weighted aggregate, etc. 
The quality of an estimator $\hat f_n$ of 
$f_{\mathbf P}$ is measured by means of a loss 
function $\ell: \mathbb L_2(\mu)\times \mathbb L_2
(\mu)\mapsto \mathbb R_+$; an estimator $\hat 
f_n$ is good if its risk 
\begin{align}
    \mathbf E_{\mathbf P}\big[\ell\big(\hat f_n 
    (\mathbf{Z}^n), f_{\mathbf P}\big)\big] 
    = \int_{\mathcal{Z}^n} \ell\big(\hat f_n 
    (\mathbf{z}), f_{\mathbf P}\big)\, \mathcal{P} 
    (d\mathbf{z})    
\end{align}
is small. A widespread choice of the loss function,
used throughout this paper except in 
\Cref{section:examples},  is the squared $\ell_2$-norm
$\ell(g,h) = \|g-h\|^2_{\mathbb L_2(\mu)} =
\int_\mathcal{X}\|g(\bx)-h(\bx)\|_2^2\, \mu(d\bx)$,
$\forall g, h \in \mathbb L_2(\mu)$.

We say that the estimator $\hat f_n$ satisfies the
PAC-Bayesian bound with prior $\pi\in \mathcal{P}_1 
(\mathcal{F}_{\mathsf W})$ and temperature 
parameter $\beta>0$, if \eqref{eq:intro} is satisfied
(where the infimum in the right hand side is over all 
$p$ in $\mathcal P_1(\mathcal F_{\mathsf W})$). 
If $C_{\textsf{PB}}=1$, the bound is called exact or sharp.
When the loss function is the squared $\mathbb L_2$-norm,
the PAC-Bayesian bound reads as
\begin{align}\label{Inequality2}
    \mathbf E_{\mathbf P}[\|\hat f_n-f_{\mathbf P} 
    \|_{\mathbb L_2}^2] \leq
    C_{\textsf{PB}} \inf_{p \in \mathcal{P}_1(
    \mathcal{F}_{\mathsf W})}  \bigg\{\int_{\mathsf W} 
    \|f_{\bw}-f_{\mathbf P}\|_{\mathbb L_2(\mu)}^2 
    \,p(d\bw) + \frac{\beta}{n} \,D_{\textsf{KL}}(p||\pi)
    \bigg\}.
\end{align}

\subsection{Shallow neural networks}

In the rest of this section, we provide more details on 
the notations and assumptions that will stand when we  
estimate $f_{\mathbf P}$ by aggregation of neural 
networks. We consider the class of networks with a 
single hidden layer and denote by $D_1$ the number 
of units in this layer.

In order to merge weights and biases of a neural network, 
we note $\bx = (1, x_1, \dots, x_{D_0-1})^\top 
\in \mathcal{X}$. The set ${\mathsf W}$ of the weights 
of a neural network can be divided into the weights of 
the hidden layer, $\bw_1$, and the weights of the output 
layer, $\bw_2$ so that  $\bw_1 \in \mathbb{R}^{D_0\times
D_1}$ and   $\bw_2 \in \mathbb{R}^{D_1\times D_2}$. 
Therefore, $\bw = (\bw_1, \bw_2)^\top$  can be seen as 
an element of $\mathbb{R}^{d}$  with the overall dimension 
$d = D_0D_1 + D_1D_2$. The neural network parametrized by $\bw$ has the form:
\begin{align}\label{NNrep}
    f_{\bw}(\bx) = \bw_2^\top\bar\sigma(\bw^\top_1 \bx) \in \mathbb R^{D_2},  \  \forall \bx\in \mathbb R^{D_0}
    \quad
    \text{with}\quad 
    \bar\sigma: \bx\in \mathbb{R}^{D_1} \mapsto 
    \begin{bmatrix}
    \sigma(x_1) \\
    \vdots \\   
    \sigma(x_{D_1})\end{bmatrix} 
    \in \mathbb{R}^{D_1},
\end{align}
where $\sigma: \mathbb R\to \mathbb R$ is a scalar
activation function. In the next sections, we will 
consider both the case of bounded and unbounded 
activation functions in order to cover most of the 
usual ones.  We refer to the bounded  case by means 
of the following assumption.
\begin{assBound}\label{ass:bounded}
The function $\sigma$ is bounded by $M_\sigma$, 
\textit{i.e}, $|\sigma(u)|\le M_\sigma$ for all 
$u\in\mathbb R$.
\end{assBound}

Let us stress that only some of our results require
\Cref{ass:bounded}. However, all our
results will require the Lipschitz assumption 
stated below, which is satisfied by sigmoid functions 
as well as piecewise continuous functions (including ReLU). 
Without loss of generality, we will assume that the
Lipschitz constant is equal to one.

\begin{assLip}\label{ass:lip} 
For every pair of real numbers $(u,u')$, we have
$|\sigma(u)-\sigma(u')|\le  |u-u'|$.
\end{assLip}

\subsection{Spherical Gaussian prior distribution}
The prior distribution $\pi$ defined in the PAC-Bayesian 
framework can be interpreted as the initial distribution 
of the weights, or as a regulariser. We focus in this 
paper on the most natural choice of prior, the Gaussian 
distribution. Recall that the weights of a neural network 
are split into two groups: the weights $\bw_1$ of the 
hidden layer
and the weights $\bw_2$ of the output layer.  To take into account their different roles we assume the distribution over $\bw$ is a product of two spherical Gaussians with different variances. 

\begin{assGaussian}\label{ass:gaussian}\text{}
The prior $\pi$ satisfies
 $\pi = \pi_1\otimes\pi_2 = 
 \mathcal{N}(0,\rho^2_1\mathbf I_{D_0D_1})\otimes \mathcal{N}(0,
 \rho^2_2 \mathbf I_{D_1D_2})$.
\end{assGaussian}

We refer to $\pi_1$ and $\pi_2$ as the distribution 
of the hidden layer and the output layer respectively.

\section{Oracle inequalities for networks with one hidden layer and Gaussian prior}\label{section:oracleBound}

In this section, we first derive a bound for the risk of the estimator $\hat f_n$ when the prior has an arbitrary centered Gaussian distribution, and subsequently provide an oracle inequality for a carefully chosen Gaussian prior. Let $\bar\bw\in\mathsf W$ be any value of the parameter. 
Using the triangle inequality, in conjunction with the fact 
that $\sqrt{(a+b)^2+c^2}\le a+\sqrt{b^2+c^2}$, one can
infer from \eqref{Inequality2} that
\begin{align}\label{eq:ineApproxEstim}
    \Big(C_{\mathsf{PB}}^{-1}\,\mathbf E_{\mathbf P}
    \big[\|\hat f_n-f_{\mathbf P}\|_{\mathbb L_2(\mu)}^2
    \big]\Big)^{1/2} 
    \le \|f_{\bar\bw}-f_{\mathbf P}\|_{\mathbb L_2(\mu)} 
    + \textup{Rem}_n(\bar\bw)^{1/2},
\end{align}
with the remainder term given by
\begin{align}\label{eq:rem}
    \textup{Rem}_n(\bar\bw) \triangleq \inf_{
    p \in \mathcal{P}_1(\mathcal{F}_{\mathsf W})}
    \bigg\{\int_{\mathsf W} 
    \|f_{\bw}-f_{\bar\bw}\|_{\mathbb L_2(\mu)}^2
    \,p(d\bw) + \frac{\beta}{n} \,D_{\textsf{KL}}(p||\pi)
    \bigg\}.
\end{align}
Considering $f_{\bar\bw}$ as an approximator of 
$f_{\mathbf P}$, the right hand side of 
\eqref{eq:ineApproxEstim} can be seen as the sum of 
the approximation error $\|f_{\bar\bw}-f_{\mathbf P} 
\|_{\mathbb L_2(\mu)} $ and the estimation error
$\textup{Rem}_n(\bar\bw)$. The main goal of this paper 
is to analyze this estimation error and then to combine 
it with available bounds on the approximation error. 
Our approach will consist in replacing the infinimum 
over all measures $p$ by the infinimum over Gaussian 
distributions, for which mathematical 
derivations are considerably simpler.

It is well-known (see, for example, 
\citep{mcallester2003pac,AlquierP2009Pbfr,guedj2019primer}) 
that for a fixed $\bar\bw$, the infinimum in 
\eqref{eq:rem} is attained by the Gibbs distribution
\begin{align}
    p^*(d\bw) \propto \exp\Big\{-\frac{n}{\beta} 
    \|f_{\bw}-f_{ \bar\bw}\|_{\mathbb L_2(\mu)}^2\Big\} 
    \pi(d\bw).
\end{align}
Furthermore in this case,
\begin{align}
    \textup{Rem}_n(\bar\bw)   = - \frac{\beta}{n} \,\log 
    \int_{\mathsf W} \exp\Big\{-\frac{n}{\beta} \|f_{\bw} 
    - f_{\bar\bw}\|_{\mathbb L_2(\mu)}^2\Big\}\pi(d\bw).
\end{align}
This expression is often referred to as the free energy. 
The content of the rest of this section can be seen as
leveraging the variational formulation \eqref{eq:rem} for 
obtaining user-friendly upper bounds. 



\begin{proposition}\label{prop:tau}
Let  \Cref{ass:lip} 
and \Cref{ass:gaussian} be satisfied.  Recall that 
$d = D_0D_1 + D_1D_2$ is the number of weights of the 
neural network and $n$ is the sample size.
\begin{itemize}\itemsep=7pt
    \item[\textup{i)}] If \Cref{ass:bounded} holds true, then 
    \begin{align}\label{eq:ineqTau}
        \textup{Rem}_n(\bar\bw) \le  \frac{\beta}{2n}
        \bigg\{\frac{\|\bar\bw_1\|_{\mathsf F}^2}{\rho_1^2 } + 
        \frac{\|\bar\bw_2\|_{\mathsf F}^2}{\rho_2^2 }+ d\log 
        \left(1+ \frac{ 2n (A_1\rho_1^2 + A_2 \rho_2^2)}{d\beta}
        \right)\bigg\}
    \end{align}
    where $A_1 = D_0D_1M_2^2\|\bar\bw_2
    \|^2_{\mathsf F}$ and  $A_2 = D_1D_2\mu(\mathcal X) 
    M_\sigma^2$.
    \item[\textup{ii)}] If the activation function is unbounded but vanishes
    at the origin, then
    \begin{align}\label{eq:ineqTau2}
        \textup{Rem}_n(\bar\bw) \le  \frac{\beta}{2n}
        \bigg\{\frac{\|\bar\bw_1\|_{\mathsf F}^2}{\rho_1^2 } + 
        \frac{\|\bar\bw_2\|_{\mathsf F}^2}{\rho_2^2 }+ 
        2d\log \left(1+ \frac{n (A_1\rho_1^2 + A_2'\rho_2^2 + A_3'
        \rho_1^2\rho_2^2)}{d\beta}\right)
        \bigg\}
    \end{align}
    where $A'_2 = \bar M_2^2\|\bar\bw_1\|_{\mathsf F}^2 
    D_2$ and $A'_3 = M_2^2D_0D_2$. 
\end{itemize}
\end{proposition}

Quantities $A_1$, $A_2$, $A_2'$ and $A_3'$ defined in the 
proposition, are independent of the sample size $n$, the temperature 
parameter $\beta$ and the variances $\rho_1$ and $\rho_2$ of the 
prior distribution, but they are dimension dependent. 

There are two dual ways of drawing statistical insights from 
the above bounds on the estimation error. The first way is to 
consider $\tau_1,\tau_2$ and $D_1$ as ``tuning parameters'' of 
the algorithm, and to prove that for a suitable choice of these
parameters the predictor $\hat f_n$ is optimal. This line of
thought is further developed in \Cref{section:riskbound} below.
The second way of interpreting the obtained bound is to 
see which functions are well estimated by $\hat f_n$ based
on $\tau_1,\tau_2$ and $D_1$. This leads to the following 
result. 

\begin{theorem}\label{prop:rho}
Let $\hat f_n$ be a method of aggregation of shallow
neural networks $\mathcal F_{\mathsf W} = \{f_{\bw} 
(\bx) = \bw_2^\top \bar\sigma(\bw_1^\top \bx) : \bw_1 
\in\mathbb R^{D_0\times D_1}; \bw_2\in\mathbb
R^{D_1\times D_2}\}$, based on a prior distribution
$\pi$, satisfying \textup{PAC}-Bayes
bound~\eqref{Inequality2}. Let Assumptions
\ref{ass:lip}  and \ref{ass:gaussian} be satisfied. 
Then, for $B_\ell = \rho_\ell\sqrt{2D_{\ell-1}D_\ell}$, $\ell = 1,2$, we have 
\begin{align}\label{main_risk_bound}
    \Big(C_{\mathsf{PB}}^{-1}\mathbf E_{\mathbf P}
    \big[\|\hat f_n-
    f_{\mathbf P}\|_{\mathbb L_2(\mu)}^2\big]\Big)^{1/2} 
    \le \inf_{\substack{\|\bw_1\|_{\mathsf F}\le B_1\\
    \|\bw_2\|_{\mathsf F}\le B_2}}
    \|f_{\bar\bw}-f_{\mathbf P}\|_{\mathbb L_2(\mu)} 
    + \Big\{\frac{\beta d}{n} \log
    \Big(3 + \frac{nE}{d\beta}\Big)\Big\}^{1/2},
\end{align}
where the constant E is defined by
\begin{align}
    E = 
    \begin{cases}
    3B_2^2 (B_1^2M_2^2 + \mu(\mathcal X)M_\sigma^2),& 
    \text{if $\sigma$ satisfies \Cref{ass:bounded}}, \\
    3B_1^2B_2^2(M_2^2 + \bar M_2^2 /D_1), & 
    \text{if $\sigma$ is unbounded but $\sigma(0) = 0$}.
    \end{cases}
\end{align}
\end{theorem}

An important consequence of this result is that the
estimation error, upper bounded by the second term in 
\eqref{main_risk_bound}, is of order $\sqrt{D_1/n}$ 
(we assume that the input and the output dimensions 
are fixed and neglect logarithmic factors). This is 
similar to many non-parametric estimation methods. 
For instance, if the regression function is estimated 
by a histogram with $K$ bins, the estimation error is 
generally of order $\sqrt{K/n}$. Thus, the number of 
units in the hidden layer of a neural network plays 
the same role as the number of bins in a histogram. 
This parameter $D_1$ has to be chosen carefully, in 
order to control both the approximation error and 
the estimation error.

\section{Examples of application}\label{section:examples}

PAC-Bayes inequality is stated in \eqref{eq:intro} in a rather
general form. In this section, we provide examples of learning
problems and learning algorithms for which a version of
\eqref{eq:intro} is satisfied. 

\subsection{Fixed design regression}

Regression with deterministic design and additive errors is 
often used in nonparametric modeling. In the case of Gaussian
errors, it corresponds to the observations 
\begin{align}
    \bZ_i = f_{\mathbf P}(\bx_i) + \sigma\bxi_i, \qquad  \bxi_i 
    \stackrel{\textup{iid}}{\sim} \mathcal{N}(0, 
    \mathbf I_{D_2}), \quad i = 1, \ldots,n,
\end{align}
where $\bx_1, \dots, \bx_n$ are given deterministic points 
and $\mathcal{Z} = \mathbb R^{D_2}$. In this case, the measure $\mu$ 
is the empirical uniform distribution: $\mu = \frac{1}{n} \sum_{i = 1}^n
\delta_{\bx_i}$.

There are many results of type \eqref{eq:intro} in the literature
for regression with fixed design. In particular, 
\citep{Leung06,dalalyan2007aggregation,DalalyanT08,DalSal,dalalyan2012sparse,Dalalyan_AIHP,Rigollet2012sparse}  
established a PAC-Bayesian bound for the exponentially weighted
aggregate defined by $\hat f_n(\mathbf{Z},\bx) = \int_{\mathsf W}
f_{\bw}(\bx)\,\hat \theta_{n,\bw}(\mathbf Z)\, \pi(d{\bw})$ with 
\begin{align}\label{thetan}
    \hat\theta_{n,\bw}(\mathbf{Z}) 
        &= \frac{\exp\{-\frac{1}{\beta} \sum_{i = 1}^n 
        \|\bZ_i- f_{\bw}(\bx_i))\|_2^2\}}{\int_{\mathsf W} 
        \exp\{-\frac{1}{\beta}\sum_{i = 1}^n \|\bZ_i -
        f_{\bu}(\bx_i))\|_2^2\}\, \pi(d\bu)}.
\end{align}
Note that $\bw\mapsto\hat\theta_{n,{\bw}}$ is a probability
density on $(\mathsf W,\pi)$, often referred to as posterior
density. For precise conditions under which the exponentially
weighted aggregate $\hat f_n(\mathbf{Z},\bx)$ satisfies 
PAC-Bayes bound \eqref{Inequality2}, the interested reader 
is referred to the papers mentioned above.

\subsection{Random design regression}

In the setting of iid observations, sharp PAC-Bayes inequality
is valid for the mirror averaging (MA) estimator  \citep{learning2008Juditsky,dalalyan2012mirror,Gerchinovitz}. 
We define the estimator in the case of regression with random 
design, and briefly mention below that similar results hold
for density estimation and classification. 
Interested reader is referred to 
\citep{learning2008Juditsky,dalalyan2012mirror}
for more detailed and comprehensive account on 
the topic. Note that similar inequalities are obtained
for the $Q$-aggregation procedure \citep{Dai,LecueRig}.

The regression problem writes as in the
previous example
\begin{align}
    \bY_i = f_{\mathbf P}(\bX_i)+ \sigma\bxi_i, \qquad  
    \bxi_i\perp\hspace{-5pt}\perp \bX_i,
    \quad i = 1, \dots,n,
\end{align}
with  $\mathcal{Z} = \mathcal{X}\times \mathcal{Y}$, 
$\mathcal{X} \subset  \mathbb R^{D_0},\;\mathcal{Y} 
\subset\mathbb{R}^{D_2}$  and $(\bX_i,\bY_i)$ being 
iid. The natural choice of the measure $\mu$ here
is the marginal distribution of $\bX_i$ over $\mathcal{X}$. 

The mirror averaging procedure satisfying \eqref{eq:intro} 
takes the form
\begin{align}\label{eq:MAmixt}
    \hat f_n(\mathbf{Z},x) =  
    \int_{\mathsf W} f_{\bw}(x)\,\hat \theta_{n,\bw}^{\textup{MA}}
    (\mathbf Z)\, \pi(d{\bw}) = \frac{1}{n+1}\sum_{m=0}^n
    \int_{\mathsf W} f_{\bw}(x)\,\hat \theta_{m,{\bw}}(\mathbf Z) 
    \,\pi(d{\bw})
\end{align}
with $\hat\theta_{0,{\bw}} = 1$ and
\begin{align}\label{eq:MAest}
    \theta_{n,\bw}^{\textup{MA}}(\mathbf{Z}) = \frac{1}{n+1} 
    \sum_{m= 0}^n  \frac{\exp\{-\frac{1}{\beta}
    \sum_{i = 1}^m  Q(\bZ_i,f_{\bw})\}}{ 
    {\int_{\mathsf W} \exp\{ -\frac{1}{\beta} 
    \sum_{i = 1}^m Q(\bZ_i,f_{\widetilde\bw})\}
    \,\pi(d\widetilde \bw)}}, 
\end{align}
where $Q:\mathcal{Z}\times\mathbb L_2(\mu) \mapsto \mathbb R$ 
is a mapping satisfying some assumptions under which the 
minimizer of the loss $\ell: g \mapsto \ell(g,f)$ coincides 
with the minimizer of $g \mapsto \mathbf E_{\mathbf{P}}
\left[Q(\mathbf{Z},g)\right]$. In the case of regression, 
the mirror averaging estimator can be evaluated with the 
$\ell_2$-norm such that in \eqref{eq:MAest}, the function 
$Q$ is given by $Q(\bZ_i, f_{\bw}) = \|\bY_i- f_{\bw} 
(\bX_i)\|_2^2$.

\subsection{Density estimation}

Consider the case where the elements of $\mathbf{Z}^n = 
(\bZ_1,\dots, \bZ_n) \in \mathcal{Z}^n$ are iid random 
variables drawn from  a distribution having $f_{\mathbf{P}}$ 
as density with respect to a measure $\mu$. We aim to 
estimate $f_{\mathbf P}$ and measure the risk using 
the squared integrated error
\begin{align}
    \ell(\hat f_n, f) =\|\hat f_n- f\|^2_{\mathbb L_2(\mu)} 
    =  \int_{\mathcal{X}} \big( \hat f_n(\bx)-f(\bx) 
    \big)^2\mu(d\bx)    
\end{align}
such that the mapping $Q$ in \eqref{eq:MAest} can by defined 
by $Q(\bx,g) = \|g\|^2_{\mathbb L_2(\mu)}-2g(\bx)$. 

\subsection{Classification for $\Phi$-risk}

Consider the binary classification problem with 
$\mathcal{Z} = \mathbb R^{D_0}\times\{-1,+1\}$
and assume that such that $\mathbf{Z}^n =((X_1,Y_1),\dots, 
(X_n,Y_n))$ are iid observations drawn from a 
distribution $\mathbf{P}$ on $\mathcal Z$. For a twice 
differentiable convex function $\Phi$, the $\Phi$-risk 
of a classifier $g: \mathbb R^{D_0}\to \{-1,+1\}$ 
is given by $R_{\mathbf P}^\Phi[g] = \mathbf E_{\mathbf 
P}[\Phi(-Yg(X))]$. In this setting, the loss function
can be defined as $\ell (g, f) = R_{\mathbf P}^\Phi[g] 
- R_{\mathbf P}^\Phi[f]$, and the MA estimator given 
by \eqref{eq:MAmixt}--\eqref{eq:MAest} can be used
with the function $Q(z,g) = \Phi(-yg(x)) $.

\section{Approximation bounds}\label{section:approxBound}

The goal of this section is to review existing bounds 
on the approximation error of neural networks for 
different classes of functions. We are particularly 
interested in shallow networks and in bounds having 
explicit dependence on the width of the hidden layer. 
The main question of interest is the assessment of 
the distance between a given function and it's best 
approximation by a one-hidden-layer network with 
$D_1$ units in the hidden layer. Our focus is on 
Lipchtiz activation functions such as logistic, tanh, 
ReLU or quadratic (for bounded inputs). Because of 
major differences between the sigmoidal and ReLU 
activation functions, these two cases will be 
presented separately.

\subsection{Bounds for sigmoidal activation functions}

For sigmoid activation functions we distinguish the 
probabilistic approach \citep{barron1993universal, Delyon1995,maiorov2000near, maiorov2006approximation} from the deterministic and 
constructive approaches \citep{mhaskar1994dimension, petrushev1998approximation,burger2001error,cao2008estimate, costarelli2013approximation,costarelli2013multivariate}. 
For the set of univariate locally $\alpha$-Hölder continuous
functions with $\alpha\in(0,1]$,  the constructive approach 
of \citep{cao2008estimate} leads to an approximation error 
of order of $D_1^{-\alpha}$ in $\ell_\infty$-norm. For 
$\alpha>1$, \citep{costarelli2013approximation} shows that 
the approximation error is $O(D_1^{-1})$ both for 
univariate and multivariate functions.

For other classes of functions, a common feature of the 
results is the requirement of the existence of some type 
of integral transform (\textit{e.g.}, Fourier, 
Radon, wavelet) of the function $f_{\mathbf P}$. Each 
transform is tailored to a different
``smoothness'' class. An early example is the
constructive approach from 
\citep{mhaskar1994dimension} that focused on 
$2\pi$-periodic functions from $\mathbb
L_2([-\pi,\pi]^{D_0})$ with absolutely convergent
Fourier series. For such functions, the approximation 
error is shown to be $O(D_1^{-1/2})$. In the case
of random design, the seminal paper \citep{barron1993universal} 
established the upper bound $O(D_1^{-1/2})$ for 
functions $f$ satisfying $\int_{\mathbb R^{D_0}} 
\|\bz\|_2 |\mathscr F[f](\bz)|\,d\bz<\infty$, with 
$\mathscr F[f]$ being the Fourier transform of $f$. 

Note that in papers mentioned in previous paragraph, the
smoothness of the function and the dimension of the 
input variable do not appear in the error bound. 
In contrast with this, for Sobolev spaces,  
\citep{petrushev1998approximation} showed how the dimension 
of the input space and the smoothness of the Sobolev space 
impact the approximation. 
Further, building on \citep{Delyon1995},  
\citep{maiorov2000near,maiorov2006approximation} 
proved that the approximation error is $O(D_1^{-r/D_0})$ 
up to a $\log(D_1)$-factor.  We use the results of 
\citep{maiorov2000near} and  \citep{maiorov2006approximation} 
to upper bound the approximation error in \eqref{main_risk_bound}. 
For $f \in \mathbb (L_2 \cap \mathbb L_1)(\mathbb R^{D_0})$ 
with Fourier transform $\mathscr F[f](\bz) = (2\pi)^{-D_0/2} 
\int_{\mathbb{R}^{D_0}} f(\bx)\,e^{i\bz^\top\bx} d\bx$, 
we define $\mathsf 
D^\balpha f = \mathscr F^{-1}[\bz^\balpha\hat f (\bz)]$. 
The unit Sobolev ball of smoothness $r$ is then 
\begin{align}
    W^{r}_2([0,1]^{D_0},\mu) = \Big\{ f: \max_{0\le 
    |\balpha| \le r} \|\mathsf D^\balpha f\|_{\mathbb L^2
    (\mu)} \le 1\Big\}.
\end{align}
To present the precise statement of the result, let 
$ \varphi,\psi\in \mathbb L_2(\mathbb R) \cap \mathbb 
L_1(\mathbb R)$ be functions satisfying
\begin{equation}\label{eq:MaiorovCondition}
     \int_0^\infty \frac{1}{a}\,\mathscr F[{\varphi}](a z)\, 
     \overline{\mathscr F[\psi]}(az)\,da = 1, \; \forall z.
\end{equation}
We define $\Phi^r$  as the set of all functions $ 
\varphi\in \mathbb L_2(\mathbb R)\cap \mathbb L_1(\mathbb R)$ such that there exists  $\psi$  satisfying \eqref{eq:MaiorovCondition} and $\forall \rho \in [0,r], \; D^\rho \varphi\in \mathbb L_2(\mathbb R),\; D^{-\rho} \psi \in \mathbb L_1(\mathbb R) $. 

\begin{theorem}[ \citep{maiorov2006approximation}, Theorem 2.3]\label{th:approxMaiorov}
Let $\mu$ be a measure with a bounded density w.r.t.\  
the Lebesgue measure and let $\sigma$ be any sigmoid function such that the function $\varphi(t) = \sigma(t+1) - \sigma(t) 
\in\Phi^r$. Then, for any $f\in W_{2}^r([0,1]^{D_0},\mu)$, 
there exists a neural network  $f_{\bw^*}$ defined as in 
\eqref{NNrep} such that
\begin{align}
\|f-f_{\bw^*}\|_{\mathbb{L}_2(\mu)}\le c_1 B_\varphi {\log(D_1)}{D_1^{-r/D_0}}
\quad\text{and}\quad
|f_{\bw^*}(\bx)|\le c_2 B_\varphi D_1^{(\frac{1}{2}-\frac{r}{D_0})_+},\; 
\forall \bx \in [0,1]^{D_0},
\end{align}    
where  $c_1$ and $c_2$ are  constants depending only on the problem dimension $D_0$ and on the regularity parameter $r$, 
whereas $B_\varphi = \max_{\rho\in [0,r]}\left\{\|D^{\rho}\varphi\|_{\mathbb L_2(\mathbb R)},\|D^{-\rho}\psi\|_{\mathbb L_1(\mathbb R)}\right\}$.
\end{theorem}

Examples of functions $\varphi$ satisfying  
\eqref{eq:MaiorovCondition} are given in 
\citep{maiorov2000near,maiorov2006approximation} without a
detailed analysis of the properties of the resulting
function $\sigma$. The next result, proved in  \Cref{ap:proofmaiorov}, fills this gap for the example 
$\varphi(x) = \frac{1}{\sqrt{2}}e^{-{x^2}/{2}}$. This
function satisfies \eqref{eq:MaiorovCondition} with 
$\psi(x) = \frac{1}{\sqrt{2}}(1-x^2)e^{-{x^2}/{2}}$. 

\begin{lemma}\label{prop:sigmaiorov} 
Let $\varphi(x) = ({1}/{\sqrt{2}}) e^{-x^2/2}$ and define 
$\sigma: \mathbb R \mapsto \mathbb R$ by    
$\sigma(x) = \sum_{j = 1}^\infty \varphi(x-j)$. This function
$\sigma$ is 1-Lipschitz continuous, nonnegative, bounded 
from above by $2.5$ and $\lim_{x \rightarrow - \infty} 
\sigma(x) = 0$. 
\end{lemma}

\begin{figure}
    \centering
    \includegraphics[width = 0.32\textwidth]{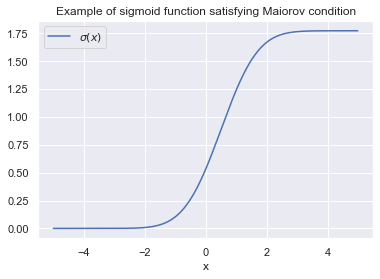}
    \includegraphics[width = 0.32\textwidth]{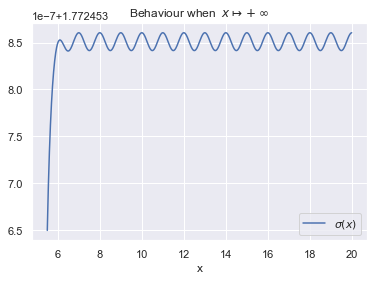}
    \includegraphics[width = 0.32\textwidth]{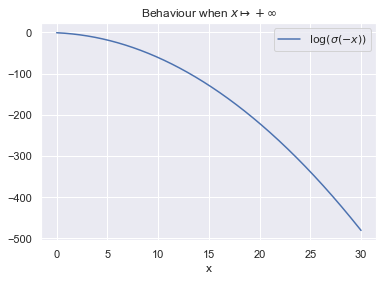}
    \caption{Activation function satisfying the Maiorov
    condition with $\varphi(x) = ({1}/{\sqrt{2}})e^{-{x^2}/{2}}$}
   \label{Fig:MaiorovCond}
\end{figure}
In \Cref{Fig:MaiorovCond} we display the function 
$\sigma$ defined in \Cref{prop:sigmaiorov}, as well 
as its limit behaviour when $|x| \rightarrow +\infty$. 
The left plot shows that $\sigma$ looks very much like 
a standard sigmoid function. The middle and the 
right plot zoom on the limit behavior at $+\infty$ 
and $-\infty$, respectively. We see, in particular, 
that $\sigma$ is not monotone when its values get close
to its upper limit, but that it is bounded everywhere
and tends exponentially fast to $0$ at $-\infty$. 
We can also consider the case where 
$\varphi(x) = \frac{(1-|x|)_+}{3}$, for which we 
displayed, in \Cref{Fig:MaiorovCond2}, the 
corresponding activation function $\sigma$. The 
function $\sigma$ is derived using the same 
methodology as for the case of \Cref{prop:sigmaiorov} 
(see also \Cref{ap:proofmaiorov}).

\subsection{Bounds for the ReLU activation function}

The literature on neural networks with ReLU 
activation has significantly grown these last years 
thanks to the computational benefits of considering 
piecewise linear activation functions 
\citep{yarotsky2017error,yarotsky2018optimal,
yarotsky2019phase,guhring2020error,lu2020deep,
shen2019deep}. We review below the results
concerning shallow networks only, leaving aside
the rich literature on approximation properties 
of deep networks.  

For a Lipschitz function $f$, approximation 
error of order $O(\eta D_1^{-{1}/{D_0}})$ is obtained in
\citep{bach2017breaking}. Following the seminal work  
\citep{MAKOVOZ199698}, results for Barron spectral 
spaces were developed in \citep{Klusowski2016UniformAB,
xu2020finite,siegel2020high}. Let $\Omega \subset 
\mathbb R^{D_0}$ be a bounded domain and $s>0$. The 
Barron spectral space of order $s$ on $\Omega$ is 
\begin{align}\label{def:Barron}
    \mathscr{B}^s(\Omega) := \Big\{f : \Omega 
    \mapsto \mathbb{C}: \|f\|_{\mathscr{B}^s(\Omega)} 
    := \inf_{f_e|\Omega = f} \int_{\mathbb{R}^d}
    \left(1+\|\bz\|_2\right)^s|
    \mathscr F[f_e] (\bz)|d\bz < \infty\Big\},
\end{align}
where $f_e$ is an $L^1(\mathbb R^d)$ extension of $f$. 
It was shown in \citep{Klusowski2016UniformAB} that the
approximation error over $\mathscr{B}^2([0,1]^{D_0})$  
is $O(D_1^{-{(D_0+2)}/{(2D_0)}})$. The same was proved 
to hold \citep{xu2020finite} for ReLU${}^k$ activation 
defined as $\sigma^{(k)}(x) = \max(0,x)^k$, when the 
target function is in $\mathscr{B}^{k+1} 
([0,1]^{D_0})$. Very recently, \citep{siegel2020high} 
made another step forward to assess the approximation error
of shallow neural networks. This result being, to the best 
of our knowledge, the tightest one for shallow networks with 
ReLU${}^k$ activations, we provide its statement in the
particular case of $k=1$.

\begin{theorem}[\citep{siegel2020high}, Theorem 3]
\label{th:SiegelXu}
Let $\Omega = [0,1]^{D_0}$ and $s \ge {1}/{2}$. 
If $f\in \mathscr{B}^s(\Omega)$ and $D_1\ge 2$, then
\begin{align}
    \|f-f_{\bw^*}\|_{\mathbb L_2(\Omega)} \le C 
    \|f\|_{\mathscr{B}^s(\Omega)} D_1^{-K}\log^m(D_1),
\end{align}
where $C$ is a constant depending on $s$ and $D_0$ 
(but not on $D_1$), whereas $K$ and $m$ are given by 
\begin{align}\label{eq:K}
    K = 
    \begin{cases}
        2 & \text{if} \; 2s\ge {3}D_0 + 4 \\
        \frac{1}{2}+ \frac{2s-1}{2(D_0+1)} & \text{if}\; 
        2s< {3}D_0 + 4
    \end{cases}
    \qquad\text{and}\qquad 
    m = 
    \begin{cases}
        0& \text{if} \; 2s < 3D_0 + 4 \\
        1 & \text{if} \; 2s > 3D_0 + 4 \\
        \frac{5}{2} & \text{if}\; 2s = 3D_0 + 4.
    \end{cases}
\end{align}

\end{theorem}

Note that in the papers summarized in this section, the 
values of the constants---that may depend on the input 
dimension and on the smoothness---are not specified. 
A unfortunate consequence of this is that we can not 
keep track of the information on the role of the input 
dimension in the risk bounds stated in the next section.

\section{Worst-case risk bounds over 
smoothness classes}\label{section:riskbound}

This section is devoted to upper bounds on the minimax risk. 
We present risk bounds for networks with sigmoid activation 
functions and prior to treating the case of ReLU activation.

\subsection{Sigmoid activation functions}

In this section we focus on real valued 
functions ($D_2=1$) belonging to the 
unit ball of the Sobolev space, $f_{\mathbf P} \in
W_2^r([0,1]^{D_0}, \mu)$. Using \Cref{prop:rho} and
\Cref{th:approxMaiorov}, 
we can express both the estimation and 
the approximation error as functions of the size
$D_1$ of the hidden layer for an activation 
function that satisfies the conditions of 
\Cref{th:approxMaiorov}. This leads to the risk bound
\begin{align}\label{eq:generalisationBounnd}
    C_{\mathsf{PB}}^{-1}\mathbf E_{\mathbf P}\big[
    \|\hat f_n-f_{\mathbf P}\|_{\mathbb L_2(\mu)}^2\big]
     \le 2c_1^2 B^2_\varphi 
    \frac{\log^2(D_1)}{D_1^{2r/D_0}}  + 
    \frac{4\beta D_1 D_0}{n} \log\Big(3 + 
    \frac{nE}{d\beta}\Big),
\end{align}
where $c_1 B_\varphi$ is as defined in \Cref{th:approxMaiorov}, 
and $E$ is defined in \Cref{prop:rho}. 
We clearly see that $D_1$, the width of the hidden 
layer,  controls the extent to which finer structure 
can be modeled. Reducing $D_1$ decreases the estimation
error since we have fewer parameters to estimate. But 
it increases the approximation error since we use 
a narrower class of approximators. Our goal below is
to determine the value of $D_1$ guaranteeing the best 
trade-off between approximation and estimation errors.

\begin{theorem}[Sigmoidal activation and Sobolev balls]
\label{prop:riskbound} 
Let $\mathcal X = [0,1]^{D_0}$ and $r>0$. Let $\hat f_n$ 
be an aggregate of neural networks satisfying PAC-Bayes
risk bound \eqref{Inequality2}. If the measure $\mu$ and 
the activation function $\sigma$ satisfy conditions of \Cref{th:approxMaiorov}, then the choice
\begin{align}\label{D1} 
    D_1 = \Big(\frac{\beta D_0}{n}\Big)^{-
    \frac{D_0}{2r+D_0}}
\end{align}
leads to the worst-case risk bound
\begin{align}\label{eq:bestboundGeneral}
   \sup_{\mathbf P: f_{\mathbf P} \in 
   W_2^r(\mathcal{X}, \mu)}\mathbf E_{\mathbf P}\big[
   \|\hat f_n-f_{\mathbf P}\|_{\mathbb L_2(\mu)}^2
   \big] 
    &\le g(n) \Big(\frac{\beta D_0}{n}\Big)^{{2r}/(2r+D_0)},
\end{align}
where $g(n)$ is the slowly varying function
\begin{align}
    g(n) = 2C_{\mathsf{PB}} \bigg(c_1^2 B_\varphi^2
    \log^2(n/\beta) + 2\log \Big(3 + \frac{3nB_2^2(B_1^2M_2^2 
    + \mu(\mathcal X)M_\sigma^2)}{d\beta}\Big)\bigg).
\end{align}
\end{theorem}

The proof of this theorem consists in substituting 
$D_1$ in \eqref{eq:generalisationBounnd}  by its expression 
\eqref{D1}. The obtained rate, $n^{-2r/(2r+D_0)}$, 
is the classical minimax rate of estimation over $D_0$-variate 
and $r$-smooth functions. We further discuss this result and 
compare it to prior work in \Cref{sec:discussion}.  

\subsection{ReLU activation function}

In the case of ReLU activation, we will state risk bounds
two classes: the Barron spectral space and a specific 
Sobolev ball. Let us first assume that $f_{\mathcal{P}} \in \mathscr{B}^s([0,1]^{D_0})$ and that the conditions of 
\Cref{th:SiegelXu} are satisfied. In view of 
\Cref{prop:rho} and \Cref{th:SiegelXu}, 
we have the risk bound
\begin{align}\label{eq:riskboundReLU}
    \mathbf E_{\mathbf P} \big[\|\hat f_n - f_{\mathcal
    P}\|_{\mathbb L_2(\mu)}^2\big]  
    &\le 2C_{\textsf{PB}}C^2 D_1^{-2K(s,D_0)}\log^{2m}(D_1) + 
    C_{\textsf{PB}}\frac{4\beta D_1 D_0}{n}\log
    \Big(3 + \frac{nE}{d\beta}\Big)
\end{align}
 where $C$, $K = K(s,D_0)$, $m$  are as in \Cref{th:SiegelXu} 
 and $E$ is as in \Cref{prop:rho}. The bias-variance balance
 equation takes the form  $D_1^{-2K} \asymp {\beta D_1}/{n}$ 
 and leads to the following proposition.

\begin{proposition}[ReLU activation and Barron spectral spaces]\label{prop:ReLUBarron}
Let $K = K(s,D_0)$ be as in \eqref{eq:K} and let 
$\hat f_n$ be an aggregate of neural networks 
satisfying PAC-Bayes risk bound \eqref{Inequality2}. 
If the conditions of \Cref{th:SiegelXu} hold then the  
choice $D_1 = \big({\beta D_0}/{n}\big)^{-1/(2K+1)}$ 
leads to the risk bound 
\begin{align}\label{eq:bestriskboundReLU}
  \sup_{\|f_{\mathbf P}\|_{
  \mathscr{B}^s([0,1]^{D_0})}\le 1}\mathbf E_{\mathbf P} 
  \big[\|\hat f_n - f_{\mathbf P}\|_{\mathbb L_2(\mu)}^2
  \big] &\le \bar g(n) \Big(\frac{\beta D_0}{n}\Big)^{
  {2K}/(2K+1)},
\end{align}
with the slowly varying function 
\begin{align}
    \bar g(n) &= 2C_{\mathsf{PB}} C^2\log^{2m}
    (n/\beta)+ 4C_{\mathsf{PB}}\log\left(3+
    \frac{3nB_1^2B_2^2(M_2^2 + \bar M_2^2)}{d\beta}\right)    
\end{align}
and $C$ is a constant that depends on $s$ and $D_0$ but not on $D_1$.
\end{proposition}

To get a risk over Sobolev spaces, we can rely on  
the inclusion $W^{s+ \nicefrac{D_0}{2} + \varepsilon,2}
(\mathcal{X}) \subset \mathscr{B}^s(\mathcal{X})$, 
true for arbitrarily small $\varepsilon>0$ 
\citep[Lemma 2.5]{xu2020finite}. This is equivalent 
to $W^{r,2}(\mathcal{X}) \subset \mathscr{B}\/^{\bar r - \nicefrac{D_0}{2}}(\mathcal{X})$ 
for every $r$, $\bar r$ such that $D_0/2\le \bar r< r$. 
Depending on the order of the Barron spectral space and 
the dimension of the problem, this might require a 
significant level of smoothness for the function $f$ we 
want to approximate. Keeping this constraint in mind, we 
proceed with the next proposition which is more easily 
comparable to \Cref{prop:riskbound}.

\begin{proposition}[ReLU activation and Sobolev space]
\label{prop:ReLUSobolev} 
Let $r\in ({D_0}/{2}, 2D_0+2)$ and let  
$\hat f_n$ be an aggregate satisfying \eqref{Inequality2}. 
For every $\bar r<r$ there is a slowly varying function 
$g_{\bar r}:\mathbb N\to \mathbb R_+$ such that
\begin{align}\label{eq:sobolevriskboundReLU}
    \sup_{f_{\mathbf P}\in W_2^r
    ([0,1]^{D_0})} \mathbf E_{\mathbf P}\big[\|\hat f_n-f_{\mathbf P}\|_{\mathbb L_2(\mu)}^2\big] 
    &\le g_{\bar r}(n) n^{-\frac{2\bar r}{2\bar r+D_0+1}}.
\end{align}
\end{proposition}

This result is weaker than the one of 
\Cref{prop:riskbound} in three aspects. First, it has the 
constraint $r\in ({D_0}/{2}, 2D_0+2)$ limiting the order 
of smoothness of Sobolev classes. The constraint 
$r< 2D_0+2$ stems from the fact that we want $K(s,D_0)$ to take the
value $(2s+D_0)/(2D_0+2)$. If $r\ge 2D_0+2$,  the claim 
of the last proposition holds true if we replace $2\bar r/
(2\bar r + D_0+1)$ by $4/5$. The second weakness 
is that $\bar r$, present in the rate of convergence, is
strictly smaller than the true smoothness $r$.  Finally, 
the denominator in the exponent has an additional 
term  increasing the dimension $D_0$ by 1, leading thus to a 
slightly slower rate of convergence than the minimax rate over 
Sobolev balls. This is a direct consequence of approximation 
properties of ReLU neural nets in Sobolev spaces. 
 
\subsection{Related work on risk bounds for 
(penalized) ERM neural networks }\label{sec:discussion}

For shallow neural networks with sigmoid activation, 
\citep{barron1994approximation} showed that the risk of 
the suitably penalized empirical risk minimizer (ERM) 
is $O(n^{-1/2}\log n)$, provided that the function $f$ is
very smooth ($\int \|\bz\|_1 |\mathscr F[f](z)|\,dz<\infty$). This was improved to $O \left(n^{-{2\bar r}/{(2 \bar r + D_0 
+ 5 )}}\right)$, $\forall \bar r<r$,  for specific cosine 
activation \citep{mccaffrey1994convergence}. 
To our knowledge, this is the best known result for a 
one-hidden-layer network provided by ERM. In the case of 
two-hidden-layer networks with sigmoid activation, the rate
$O(n^{-2r/(2r+D_0)}\log^3 n)$ was obtained in
\citep{bauer2019deep} for functions satisfying a generalized
hierarchical interaction model. 
Our risk bound \eqref{eq:bestboundGeneral}, of order
$O(n^{-2r/(2r+D_0)}\log n)$, matches the nonparametric
minimax rate \citep{stone1982optimal, 
tsybakov2008introduction}, and is better than known rates 
for the ERM networks with one hidden layer. Roughly 
speaking, this shows that 
aggregation acts as an additional layer, so that the 
aggregated one-hidden-layer networks achieve the
same rate as the ERM two-hidden-layer networks.

We switch now to neural networks with ReLU activation
functions. For one-hidden-layer networks, 
\citep{bach2017breaking} established a risk bound of order $n^{-{2}/{(D_0+3)}}$. On a related note, \citep{klusowski2016risk} considered bounded ramp activation functions and the low
dimensional setting $D_0\ll n$. For functions belonging
to $\mathscr{B}^2([0,1]^{D_0})$, they proved that the 
risk of the penalized ERM is $O(n^{-{(D_0+4)}/{(2D_0+6)}})$. 
This result can be directly compared to ours,  in the 
particular case $s = 2$; \Cref{prop:ReLUBarron} and the 
fact that ${2K}/{(2K+1)}= {(2s+D_0)}/{(2s+ 2D_0+ 1)}  
= (D_0 + 4)/(2D_0+5)$, yield a leading term of order
$O\big(n^{-(D_0+4)/(2D_0+5)} \big)$. This improves the 
result of \citep{klusowski2016risk} by a factor $O( n^{-\nicefrac{(D_0+4)} 
{2(3+D_0)(2D_0+5)}})$. For instance, if $D_0 = 3$ or 
$D_0 = 4$, we get the improvement factors $n^{-7/132}$ 
and $n^{-4/91}$, respectively. This improvement 
vanishes when $D_0$ increases to infinity. 
For multilayer ReLU networks, \citep{schmidt2020nonparametric}
established the counterpart of the risk bound of 
\citep{bauer2019deep} for $\beta$-Hölder functions. In 
particular, the worst-case risk was shown 
to be $O(n^{-{2\beta}/{(2\beta+D_0)}})$, see also 
\citep{suzuki2018adaptivity} for an analogous 
result over Besov spaces. 
Hence, the minimax rate is achieved by the 
ERM over multilayer ReLU networks. In view of  \Cref{prop:ReLUSobolev},  this provides a bound for the
ERM over multilayer networks smaller by a factor 
$O(n^{-{2\bar r}/ (2\bar r+D_0+1)
(2\bar r+D_0)})$ then the bound for the aggregate
of one-hidden-layer networks.
 
 \section{Conclusion and outlook}\label{section:conclusion}

We have analyzed the estimation error of an aggregate of
neural networks having one hidden layer and Lipschitz continuous
activation function, under the condition that the aggregate
satisfies the PAC-Bayes inequality. We focused our attention
on Gaussian priors and obtained risk bounds in which the
dependence on all the involved parameters is explicit. All 
these bounds on the estimation error come with explicit constants. We then combined our bounds on the estimation 
error with bounds on approximation error available in the
literature. This allowed us to prove that aggregation of
one-layer neural networks achieves the minimax risk over
conventional smoothness
classes. On the down side, since the constants in the bounds 
on the approximation error available in the literature 
are not explicit, the same is true for risk bounds of the present
work. Therefore, it would be highly relevant to refine the
existing approximation bounds to make appear all
the constants.

The results of the present work can be extended in different
directions. First, it would be interesting to consider the
problem of aggregation of deep neural networks in order to
understand possible benefits of increasing the depth. Second, 
it might be relevant to analyze the case of a prior with 
heavier tails, such as the Laplace prior or the Student prior,
with a hope to cover the case of high dimension $D_0>n$ under
some kind of sparsity assumption. Finally, another avenue of
future research is to explore the computational benefits of
considering aggregated neural networks in conjunction with
the Langevin-type algorithms.

\vfill
\newpage

\bibliography{bibliography.bib}

\begin{appendix}
\section{Proofs}\label{section:appendix}

As a preliminary remark let us note that, as a mixing measure, we expect the distribution $p$ to aggregate the predictors $f_{\bw}$ so that the resulting estimator is almost as good as the best predictors in $\mathcal{F}_{\mathsf W}$. 
A direct consequence of it is that ``a good choice'' 
of $p$ should be centered in $\bar\bw$. This is an 
heuristic way to choose the mean, and all along the
appendix we will fix the distribution of $p$ as
\begin{align}\label{ass:bwp}
 p = p_1\otimes p_2  \sim \mathcal{N}(\bar\bw_1,\tau_1^2
 \mathbf I_{D_1D_0})\otimes \mathcal{N}
 (\bar\bw_2,\tau_2^2 \mathbf I_{D_1D_2}), \qquad  
 \tau_1, \tau_2 >0
\end{align}
 where   $\bar\bw \in \text{argmin}_{\bw\in\mathsf W} \|f_{\bw}-f_{\mathbf P}\|_{\mathbb L_2(\mu)}$.
The additional condition \eqref{ass:bwp} is the starting point of our choice for $p$, it is now left to set values for the variance $(\tau_1^2, \tau_2^2)$.  

\subsection{Some useful lemmas}\label{ap:lemmas}

In what follows, when appropriate, we will write $f_{\bw_1,\bw_2}$ instead of $f_{\bw}$.

\begin{lemma}\label{lem:decomposition}
If the probability distribution $p$ is such that 
$p(d\bw) = p_1(d\bw_1)p_2(d\bw_2)$ with 
\begin{align}
    \int_{\mathsf W_2} \bw_2 p_2(d\bw_2) = \bar\bw_2
\end{align}
then 
\begin{align}
    \int_{\mathsf W}\|f_{\bw} - f_{\bar\bw}\|_{\mathbb L_2(\mu)}^2 \,p(d\bw)
    =\int_{\mathsf W}\|f_{\bw} - f_{\bw_1,\bar\bw_2}\|_{\mathbb L_2(\mu)}^2 \,p(d\bw)
    +\int_{\mathsf W_1}\|f_{\bw_1,\bar\bw_2} - f_{\bar\bw}\|_{\mathbb L_2}^2 \,p_1(d\bw_1).
\end{align}
\end{lemma}
\begin{proof}
Simple algebra yields
\begin{align}
    \int_{\mathsf W} \big(f_{\bw} - f_{\bar\bw}\big)^2(\bx)\,p(d\bw) 
    &= \int_{\mathsf W} (f_{\bw} - f_{\bw_1,\bar\bw_2}+ f_{\bw_1,\bar\bw_2} - f_{\bar\bw})^2(\bx)\,p(d\bw) \\
    & = \int_{\mathsf W} (f_{\bw} - f_{\bw_1,\bar\bw_2})^2(\bx)\,p(d\bw) + \int_{\mathsf W} (f_{\bw_1,\bar\bw_2} - f_{\bar\bw})^2(\bx)\,p(d\bw)\\
        &\qquad + 2  \underbrace{\int_{\mathsf W} 
        (f_{\bw_1,\bw_2} - f_{\bw_1,\bar\bw_2})(\bx)(f_{\bw_1,\bar\bw_2} - f_{\bar\bw_1,\bar\bw_2})(\bx)\,p(d\bw)}_{\triangleq A}.
\end{align}
To complete the proof it suffices to integrate the 
previous equality with respect to $\mu(d\bx)$ in virtue of Fubini-Tonelli theorem and to
check that $A = 0$. The latter property follows from the
fact that $p$ is a product measure and, for all
$\bw_1\in\mathsf W_1$,
\begin{align}
    \int_{\mathsf W_2} 
        (f_{\bw_1,\bw_2} - f_{\bw_1,\bar\bw_2})(\bx)\,p_2(d\bw_2) =
    \int_{\mathsf W_2} 
        (\bw_2 - \bar\bw_2)^\top \bar\sigma(\bw_1^\top\bx)\,p_2(d\bw_2)=0.    
\end{align}
This yields the claim of the lemma.
\end{proof}
In this section and the next one, let us define the two quantities:
\begin{align}
    G_1(\bw) &=  \|f_{\bw} - f_{\bw_1,\bar\bw_2}\|_{\mathbb L_2(\mu)}^2,\qquad\text{and}\qquad
    G_2(\bw_1) =  \|f_{\bw_1,\bar\bw_2} - f_{\bar\bw}\|_{\mathbb L_2(\mu)}^2.
\end{align}

\begin{lemma}\label{lem:G2}
If Assumptions \ref{ass:lip} and $M_2<\infty$ are satisfied,
and $p$ is chosen as in \eqref{ass:bwp}, then 
\begin{align}\label{eq:G2ineq}
    \int_{\mathsf W_1} G_2(\bw_1)\,p_1(d\bw_1) 
     &\le  D_0\big(  M_2 \tau_1 \|\bar\bw_2\|_{1,2}\big)^2
     \le C_1 D_0D_1\tau_1^2
\end{align}
with $C_1 = ( M_2 \|\bar\bw_2\|_{\mathsf F})^2$.
\end{lemma}
\begin{proof}[Proof of \Cref{lem:G2}]
We first use the fact that
$\sigma$ is 1-Lipschitz. On the one hand, in conjunction with 
the Fubini-Tonelli theorem, this yields
\begin{align}
    \int_{\mathsf W_1} G_2(\bw_1)\,p_1(d\bw_1) & = 
    \int_{\mathcal X}\int_{\mathsf W} \left\|\bar\bw_2^{\top}\left\{\bar\sigma(\bw_1^\top \bx)-\bar\sigma(\bar\bw_1^{\top} \bx)\right\}\right\|_2^2p(d\bw)\mu(d\bx)\\
    & \le  \int_{\mathcal X}\int_{\mathsf W} \sum_{j=1}^{D_2} 
    \left(\sum_{i = 1}^{D_1}|\bar\bw_{2,ij}||(\bw_1 -\bar\bw_1)_i
    ^\top\bx|\right)^2p(d\bw)\mu(d\bx)\\
    & \le  \int_{\mathcal X} \sum_{j=1}^{D_2} 
    \left(\sum_{i = 1}^{D_1}|\bar\bw_{2,ij}|
    \bigg\{\int_{\mathsf W}|(\bw_1 - \bar\bw_1)_i
    ^\top\bx|^2p(d\bw)\bigg\}^{1/2}\right)^2\mu(d\bx)\\
    &= D_0M_2^2\tau_1^2 \sum_{j=1}^{D_2} 
    \left(\sum_{i = 1}^{D_1}|\bar\bw_{2,ij}|\right)^2\le D_0
    M_2^2\tau_1^2 D_1\|\bar\bw_2\|_{\mathsf F}^2
\end{align}
and the claim of the lemma follows.
\end{proof}
In view of \Cref{lem:decomposition} and \Cref{lem:G2}, 
we have
\begin{align}\label{eq:1}
    \int_{\mathsf W} \|f_{\bw} - f_{\bar\bw}\|_{\mathbb L_2(\mu)}^2\, p(d\bw) = 
    \int_{\mathsf W} G_1(\bw)\, p(d\bw) + 
    \int_{\mathsf W} G_2(\bw_1)\, p_1(d\bw_1).
\end{align}
and \begin{align}
    \int_{\mathsf W} G_2(\bw_1)\,p(d\bw) 
     &\le   D_1 ( M_2 \|\bar\bw_2\|_{\mathsf F})^2 
     \tau^2_1.
\end{align}
We now state two distinct lemmas to bound the quantity 
$\int_{\mathsf W} G_1(\bw)\, p(d\bw)$. \Cref{lem:G1} 
account for bounded activation functions whereas 
\Cref{lem:G1ReLU} focuses on unbounded ones.

\begin{lemma}\label{lem:G1}
Under  \Cref{ass:bounded} and 
$M_2<\infty$, if $p$ is given by \eqref{ass:bwp},
we have
\begin{align}
    \int_{\mathsf W} G_1(\bw)\,p(d\bw) 
    &\le (M_\sigma\tau_2)^2 \mu(\mathcal X) D_1D_2.
\end{align}
\end{lemma}

\begin{proof}[Proof of \Cref{lem:G1}]
Using Fubini-Tonelli theorem, we get
\begin{align}
\int_{\mathsf W} G_1(\bw)\,p(d\bw) 
    & =
    \int_{\mathcal X}\int_{\mathsf W_1}
    \underbrace{\int_{\mathsf W_2}
    \Big\|(\bw_2 - \bar\bw_2)^\top \bar\sigma(\bw_1^\top
    \bx)\Big\|_{\mathbb L_2(\mu)}^2 \, p_2(d\bw_2)}
    _{:=I(\bx,\bw_1)}
    p_1(d\bw_1)\,\mu(d\bx).
\end{align}
For the inner integral, simple algebra yields
\begin{align}
I(\bx,\bw_1)& =
    \int_{\mathsf W}
    \bar\sigma(\bw_1^\top \bx)^\top(\bw_2-\bar\bw_2)
    (\bw_2-\bar\bw_2)^\top \bar\sigma(\bw_1^\top \bx)\,p(d\bw)\\
    & =
    \bar\sigma(\bw_1^\top \bx)^\top\bigg\{\int_{\mathsf W_2}(\bw_2-\bar\bw_2)
    (\bw_2-\bar\bw_2)^\top\,p_2(d\bw_2)\bigg\} \bar\sigma(\bw_1^\top \bx)\\
    & = \tau_2^2D_2 \|
    \bar\sigma(\bw_1^\top \bx)\|_2^2
    \label{aux2}.
\end{align}
Therefore, 
\begin{align}
    \int_{\mathcal X}\int_{\mathsf W_1} I(\bx,\bw_1)\,
    p_1(d\bw_1)\,\mu(d\bx) \le M_\sigma^2\mu(
    \mathcal X)\tau^2_2D_1D_2.
\end{align}
This completes the proof of the lemma.
\end{proof}

\begin{lemma}\label{lem:G1ReLU}
Let $\bar M_2 = \|\int_{\mathcal X} \bx\bx^\top\mu(d\bx)\|_{\mathsf{sp}}$
be the spectral norm of the ``covariance'' matrix of the design. 
Under \Cref{ass:lip}, if $p$ is given by \eqref{ass:bwp} 
and $\sigma(0) = 0$, we have 
\begin{align}
    \int_{\mathsf W} G_1(\bw)\,p(d\bw) 
    &\le M_2^2  D_0 D_1 D_2\tau_1^2\tau_2^2 + 
    \bar M_2^2  D_2 \|\bar\bw_1\|_{\mathsf F}^2\tau_2^2.
\end{align}
\end{lemma}

\begin{proof}[Proof of \Cref{lem:G1ReLU}]
Using \eqref{aux2}, we get
\begin{align}
    \int_{\mathsf W} G_1(&\bw)\,p(d\bw) 
    =  \tau_2^2D_2\int_{\mathcal X}\int_{\mathsf W_1} \|
    \bar\sigma(\bw_1^\top \bx)\|_2^2\,p_1(d\bw_1)\,\mu(d\bx)\\
    &\leq  \tau_2^2D_2\int_{\mathcal X}\int_{\mathsf W_1} 
    \|\bw_1^\top \bx\|_2^2\,p_1(d\bw_1)\,\mu(d\bx)\\
    &=  \tau_2^2D_2\int_{\mathcal X}\int_{\mathsf W_1} 
    \|(\bw_1-\bar\bw_1)^\top \bx\|_2^2\,p_1(d\bw_1)\,\mu(d\bx) 
    + \tau_2^2D_2\int_{\mathcal X} 
    \|(\bar\bw_1)^\top \bx\|_2^2\,\mu(d\bx)\\
    & = M_2^2 D_0D_1D_2 \tau_1^2\tau_2^2
    + \bar M_2^2D_2\|\bar\bw_1\|_{\mathsf F}^2\tau_2^2.
\end{align}
This completes the proof of the lemma.
\end{proof}

\begin{lemma}\label{lem:estimationBound} Under \Cref{ass:lip} and $M_2<\infty$, 
if $p$ is given by \eqref{ass:bwp}, then
\begin{align}\label{int_bound2}
    \int_{\mathsf W} 
    \|f_{\bw} - f_{\bar\bw}\|_{\mathbb L_2(\mu)}^2\, p(d\bw) 
    \le M^2_2\|\bar\bw_2\|_{\mathsf F}^2 D_0D_1 \tau^2_1 
    + \bar M_2^2D_2\|\bar\bw_1\|_{\mathsf F}^2\tau_2^2 
    + M_2^2 D_0D_1D_2 \tau_1^2\tau_2^2.
\end{align}
If, in addition, \Cref{ass:bounded} is satisfied, then
\begin{align}\label{int_bound1}
    \int_{\mathsf W} \|f_{\bw} - f_{\bar\bw}\|_{\mathbb L_2(\mu)}^2\, p(d\bw) \le    M^2_2 
    \|\bar\bw_2\|_{\mathsf F}^2 D_0D_1 \tau^2_1 + \mu(\mathcal X) 
    M^2_\sigma D_1D_2\tau^2_2.
\end{align}
\end{lemma}

\begin{proof}
In \Cref{lem:decomposition}  we have checked that
\begin{align}
    \int_{\mathsf W} \|f_{\bw} - f_{\bar\bw}\|_{\mathbb L_2(\mu)}^2\, p(d\bw) = 
    \int_{\mathsf W} G_1(\bw)\, p(d\bw) + 
    \int_{\mathsf W} G_2(\bw_1)\, p_1(d\bw_1).
\end{align}
\Cref{lem:G1} and \Cref{lem:G2} take care of both integrals in the right hand side of the equality for bounded activation functions and we directly get
\eqref{int_bound1}. Similarly, \Cref{lem:G1ReLU} and 
\Cref{lem:G2} can be applied for unbounded activation functions, leading to \eqref{int_bound2}. 
\end{proof}



\subsection{Proof of Proposition \ref{prop:tau}}\label{ap:propOptim1}
Recall that the goal is to find an upper bound
for the remainder term
\begin{align}\label{eq:rembis}
    \textup{Rem}_n(\bar\bw) \triangleq \inf_{
    p \in \mathcal{P}_1(\mathcal{F}_{\mathsf W})}
    \bigg\{\int_{\mathsf W} 
    \|f_{\bw}-f_{\bar\bw}\|_{\mathbb L_2(\mu)}^2
    \,p(d\bw) + \frac{\beta}{n} \,D_{\textsf{KL}}(p||\pi)
    \bigg\}.
\end{align}
We start this proof by considering the case where Assumptions 
\ref{ass:lip}, \ref{ass:bounded} and \ref{ass:gaussian} are 
satisfied. We choose as 
$p$ the product of two spherical Gaussian distributions 
with variances $\tau_1^2$ and $\tau_2^2$, as specified 
in \eqref{ass:bwp}. In this case, the Kullback-Leibler divergence $D_{\textsf{KL}}(p||\pi)$ is given by 
\begin{align}\label{eq:KL}
    D_{\textsf{KL}}(p||\pi) =
    \frac{1}{2}\sum_{\ell=1}^2 \bigg\{ 
    \frac{\|\bar\bw_\ell\|_{\mathsf F}^2}{\rho_\ell^2} 
    + D_{\ell-1}D_\ell\bigg[\left(\frac{\tau_\ell
    }{\rho_\ell}\right)^2-1 -  \log\left(\frac{ 
    \tau^2_\ell}{\rho^2_\ell}\right) \bigg]\bigg\}.
\end{align}
It is now left to find good values for $\tau^2_1$ 
and $\tau_2^2$. Combining with the result \eqref{int_bound1} 
of \Cref{lem:estimationBound}, we get the inequality
\begin{align}\label{eq:rem2}
    \textup{Rem}_n(\bar\bw) \leq \frac{\beta\|\bar\bw_1\|_{\mathsf F}^2}{2n\rho_1^2} 
    +\frac{\beta\|\bar\bw_2\|_{\mathsf F}^2}{2n\rho_2^2} + 
    \frac{\beta}{2n}\sum_{\ell=1}^2 D_{\ell-1}D_\ell
    \bigg\{C_\ell\left(\frac{\tau_\ell
    }{\rho_\ell}\right)^2-1 -  \log\left(\frac{ 
    \tau^2_\ell}{\rho^2_\ell}\right) \bigg\}
\end{align}
where 
\begin{align}
    C_1 = \frac{2n M_2^2\|\bar\bw_2\|_{\mathsf F}^2 
    \rho_1^2}{\beta} +1,
    \quad
    C_2 = \frac{2n\mu(\mathcal X)M_\sigma^2 \rho_2^2}{\beta} 
    +1.
\end{align}
One can easily check that the minimum of the function
$u\mapsto Cu-1-\log u$ is attained at $u_{\min} = 1/C$ 
and the value at this point is $\log C$.  This implies 
that 
\begin{align}
    \textup{Rem}_n(\bar\bw) & \leq \frac{\beta\|\bar\bw_1\|_{\mathsf F}^2}{2n\rho_1^2} 
    +\frac{\beta\|\bar\bw_2\|_{\mathsf F}^2}{2n\rho_2^2} + 
    \frac{\beta}{2n}\sum_{\ell=1}^2 D_{\ell-1}D_\ell
    \log C_\ell\label{eq:rem2a}\\
    &\stackrel{(1)}{\le} \frac{\beta\|\bar\bw_1\|_{\mathsf F}^2}{2n\rho_1^2} 
    +\frac{\beta\|\bar\bw_2\|_{\mathsf F}^2}{2n\rho_2^2} + 
    \frac{\beta d}{2n} \log \left(\frac{D_0D_1C_1+D_1D_2C_2}{d}\right)\\
    &\le \frac{\beta\|\bar\bw_1\|_{\mathsf F}^2}{2n\rho_1^2} 
    +\frac{\beta\|\bar\bw_2\|_{\mathsf F}^2}{2n\rho_2^2} + 
    \frac{\beta d}{2n} \log \bigg(1 + \frac{2n (D_0D_1  
    M_2^2\|\bar\bw_2\|_{\mathsf F}^2 \rho_1^2 + D_1D_2\mu(\mathcal X)M_\sigma^2\rho_2^2)}{
    \beta d}\bigg),\label{eq:rem2b}
\end{align}
where in (1) we have used the concavity of the function
$u\mapsto \log u$. This completes the proof of the first 
claim of \Cref{prop:tau}.

In the case where \Cref{ass:bounded} is not fulfilled,
but instead $\sigma(0) = 0$, we repeat the same scheme
of proof as above by using \eqref{int_bound2} instead
of \eqref{int_bound1}. This leads to 
\begin{align}
    \textup{Rem}_n(\bar\bw) &\leq   
    \frac{\beta d}{2n}
    \bigg\{C'_1\left(\frac{\tau_1}{\rho_1}\right)^2 + C_2'
    \left(\frac{\tau_2}{\rho_2}\right)^2+ C_3'\left(\frac{\tau_1}{\rho_1}\right)^2\left(
    \frac{\tau_2}{\rho_2}\right)^2 - 2 -  \log\left(\frac{ 
    \tau^2_1\tau^2_2}{\rho^2_1\rho^2_2}\right) \bigg\}\\
    &\qquad+ \frac{\beta\|\bar\bw_1\|_{\mathsf F}^2}{2n\rho_1^2} 
    +\frac{\beta\|\bar\bw_2\|_{\mathsf F}^2}{2n\rho_2^2}.
    \label{eq:rem3}
\end{align}
where 
\begin{align}
    C'_1 = \frac{2n D_0 M_2^2\|\bar\bw_2\|_{\mathsf F}^2 
    \rho_1^2}{\beta (D_0+D_2)} +1,
    \quad
    C'_2 = \frac{2n \bar M_2^2\|\bar\bw_1\|^2_{\mathsf F}
    D_2 \rho_2^2}{\beta(D_0+D_2)D_1} + 1,\quad
    C_3' = \frac{2n  M_2^2 \rho_1^2\rho_2^2D_0D_2}{
    \beta(D_0+D_2)}.
\end{align}
We choose $\tau_1$ and $\tau_2$ so that
\begin{align}
    \bigg(\frac{\tau_1}{\rho_1}\bigg)^2 = \frac{1}{C_1'+C_3'(\tau_2/\rho_2)^2},\qquad
    \bigg(\frac{\tau_2}{\rho_2}\bigg)^2 = 1/C_2'.
\end{align}
With this choice of $\tau_1$ and $\tau_2$ in \eqref{eq:rem3} 
and simple algebra,  we get
\begin{align}
    \textup{Rem}_n(\bar\bw) &\leq   
    \frac{\beta d}{2n}\,\log\big(C_1'C_2'+C_3'\big) + \frac{\beta\|\bar\bw_1\|_{\mathsf F}^2}{2n\rho_1^2} 
    +\frac{\beta\|\bar\bw_2\|_{\mathsf F}^2}{2n\rho_2^2}.
    \label{eq:rem2bis}
\end{align}
To complete the proof, we use the following inequalities
\begin{align}
    \ln(C_1'C_2' + C_3') & \le \log \big(C_1'(C_2' +C_3')\big) \\
    & = \log C_1' + \log (C_2' +C_3')\\
    &\le 2 \ln((C_1'+C_2'+C_3')/2),
\end{align}
where the first inequality follows from the fact that 
$C_2'\ge 1$ whereas the last inequality is a consequence of the 
concavity of the logarithm.

\begin{remark}
The distribution $p$ is centered on the oracle choice 
$\bar\bw$ for the weights of the neural network and 
we observe that the optimized variances $(\tau_1^2, \tau_2^2)$ 
in the proof of \Cref{prop:tau} are of the form $\tau^2_l = 
{\rho^2_\ell}/{(1+ c_\ell n \rho^2_\ell)}$, $\ell = 1, 2,$ 
for some positive constants $c_1,c_2$. These values of 
$\tau_\ell$ arbitrate between the prior beliefs and the 
information brought by data. Indeed, $(1)$ when no training 
data is available the uncertainty around $\bar\bw$ corresponds 
to the prior uncertainty $(\rho^2_1, \rho_2^2)$, $(2)$ when 
the amount of observations $n$ is unlimited and goes to 
infinity the uncertainty around the oracle value converges 
to $0$ and $p$ becomes close to the Dirac mass in $\bar\bw$.
\end{remark}

\subsection{Proof of \Cref{prop:rho}} 
\label{ap:propOptim2}

The main idea is to choose $\rho_1$ and $\rho_2$ minimizing
the upper bound of the worst-case value of the remainder term
\begin{align}
    \sup_{\bar\bw : \|\bar\bw_\ell\|_{\mathsf F}\le B_\ell} 
    \text{Rem}_n(\bar\bw)
\end{align}
furnished by \Cref{prop:tau}. Instead of using the exact 
minimizer, we use a surrogate obtained by simplifying
expressions of $\rho_1$ and $\rho_2$. This is done by the 
following result. 

\begin{corollary}\label{cor:rho}
Let Assumptions \ref{ass:lip}  and \ref{ass:gaussian} 
be satisfied, set $B_\ell = \rho_\ell\sqrt{
2D_{\ell-1} D_\ell}$\/ for $\ell = 1,2$.
\begin{itemize}
    \item[\textup{i)}] If \Cref{ass:bounded} holds true, then 
    \begin{align}
        \sup_{\bar\bw : \|\bar\bw_\ell\|_{\mathsf F}\le B_\ell}\textup{Rem}_n(\bar\bw) \le \frac{\beta d}{n} 
        \log\bigg(3+ \frac{3nB_2^2 (B_1^2M_2^2 
        + \mu(\mathcal X)M_\sigma^2)}{d\beta}\bigg).
    \end{align}
    \item[\textup{ii)}] If the activation function is unbounded
    but vanishes at the origin, then
    \begin{align}
        \sup_{\bar\bw : \|\bar\bw_\ell\|_{\mathsf F}\le B_\ell}\textup{Rem}_n(\bar\bw) \le \frac{\beta d}{n}
        \log \left(3+ \frac{3nB_1^2B_2^2(M_2^2 
        + \bar M_2^2 /D_1)}{d\beta}\right).
    \end{align}
\end{itemize}
\end{corollary}

The rest of this section is devoted to the 
proof of this claim, which implies the claim 
of \Cref{prop:rho}. In view of \eqref{eq:rem2a}, 
we have
\begin{align}
    \textup{Rem}_n(\bar\bw) & \leq \frac{\beta\|\bar\bw_1\|_{\mathsf F}^2}{2n\rho_1^2} 
    +\frac{\beta\|\bar\bw_2\|_{\mathsf F}^2}{2n\rho_2^2} + 
    \frac{\beta}{2n}\sum_{\ell=1}^2 D_{\ell-1}D_\ell
    \log (1 + F_\ell \rho_\ell^2)
\end{align}
with 
\begin{align}
    F_1 = \frac{2n  M_2^2 \|\bar\bw_2\|_{
    \mathsf F}^2}{\beta}\qquad\text{and}\qquad 
    F_2 = \frac{2n \mu(\mathcal X)M_\sigma^2}{\beta}.
\end{align}
Taking the maximum over all $\bar\bw$ such that the 
Frobenius norms of $\bar\bw_1$ and $\bar\bw_2$ are 
bounded by $B_1$ and $B_2$, we get
\begin{align}\label{sup_rem}
    \sup_{\|\bar\bw_1\|_{\mathsf F}\le B_1}
    \sup_{\|\bar\bw_2\|_{\mathsf F}\le B_2}
    \textup{Rem}_n(\bar\bw) & \leq \frac{\beta B_1^2}{2n\rho_1^2} 
    +\frac{\beta B_2^2}{2n\rho_2^2} + \frac{\beta}{2n} 
    \sum_{\ell=1}^2 D_{\ell-1}D_\ell \log (1 + \bar 
    F_\ell \rho_\ell^2)
\end{align}
with 
\begin{align}
    \bar F_1 = \frac{2n  M_2^2 B_2^2}{
    \beta}\qquad\text{and}\qquad 
    \bar F_2 = \frac{2n \mu(\mathcal X)M_\sigma^2}{\beta}.
\end{align}
The first order necessary condition for optimizing 
the right hand side with respect to $\rho_1^2$ and $\rho_2^2$ 
reads as
\begin{align}\label{foc}
    -\frac{B_\ell^2}{\rho_\ell^4} + 
     \frac{D_{\ell-1}D_\ell\bar F_\ell}{1 + \bar F_\ell 
     \rho_\ell^2} = 0
     \quad \Longleftrightarrow \quad
     \rho_\ell^4 - \frac{B_\ell^2}{D_{\ell-1}D_\ell}\, 
    \rho_\ell^2 - \frac{B_\ell^2}{D_{\ell-1}D_\ell \bar F_\ell} 
    = 0.
\end{align}
This second-order equation has only one positive root given by
\begin{align}
    \rho_\ell^2 
    & = \frac{B_\ell^2}{2D_{\ell-1}D_\ell} + 
    \bigg(\frac{B_\ell^4}{4D_{\ell-1}^2D_\ell^2} + 
    \frac{B_\ell^2}{D_{\ell-1}D_\ell \bar F_\ell}
    \bigg)^{1/2}\\
    & = \frac{B_\ell^2}{2D_{\ell-1}D_\ell}\bigg\{1 + 
    \bigg(1 + \frac{4D_{\ell-1}D_\ell}{B_\ell^2 \bar F_\ell}
    \bigg)^{1/2}\bigg\}.
\end{align}
We simplify computations by choosing 
\begin{align}
    \rho_\ell^2 & = \frac{B_\ell^2}{2D_{\ell-1}D_\ell}.
\end{align}
Replacing these values of $\rho_\ell^2$ in \eqref{sup_rem},
we get
\begin{align}\label{sup_rem2}
    \sup_{\|\bar\bw_\ell\|_{\mathsf F}\le B_\ell}
    \textup{Rem}_n(\bar\bw) & \leq 
    \frac{\beta}{n}\sum_{\ell= 1} ^2 
    D_{\ell-1}D_\ell \bigg\{1 
    + \frac12 \log\bigg(1+ 
    \frac{B_\ell^2\bar F_\ell}{2D_{\ell-1}D_\ell}
    \bigg)\bigg\}\\
    &\le \frac{\beta d}{n}\sum_{\ell= 1} ^2 
    \bigg\{1 
    + \frac12 \log\bigg(1+ 
    \frac{B_1^2\bar F_1 + B_2^2 \bar F_2}{2d}
    \bigg)\bigg\},
\end{align}
where the last inequality follows from the concavity
of the logarithm. Replacing $\bar F_1$ and $\bar F_2$
with their respective expressions, we get the inequality
\begin{align}
    \sup_{\|\bar\bw_\ell\|_{\mathsf F}\le B_\ell}
    \textup{Rem}_n(\bar\bw) 
    &\le \frac{\beta d}{n} \Bigg(
    1 + \frac12 \log\bigg(1+ \frac{nB_2^2 
    (B_1^2M_2^2 + \mu(\mathcal X)M_\sigma^2)}{
    d\beta}\bigg)\Bigg)\\
    &\le \frac{\beta d}{n} \log\bigg(3+ \frac{3nB_2^2 
    (B_1^2M_2^2 + \mu(\mathcal X)M_\sigma^2)}{d\beta}
    \bigg),
\end{align}
which coincides with the first claim of the corollary. 

The second claim of the proposition is obtained by replacing
$\rho_\ell$'s by their respective expressions in the second
claim of \Cref{prop:tau}.

\begin{figure}
      \centering
      \includegraphics[width = 0.5\textwidth]{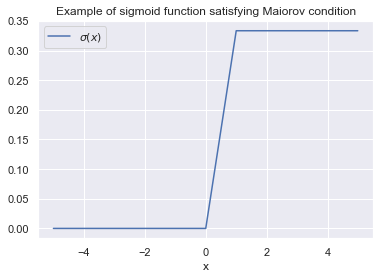}
      \caption{Sigmoid function satisfying the Maiorov condition with $\varphi(x) = {(1-|x|)_+}/{3}$.}
   \label{Fig:MaiorovCond2}
\end{figure}

\subsection{Proof of Proposition \ref{prop:sigmaiorov}}\label{ap:proofmaiorov}
    Since
    \begin{align}\label{eq:AppexSig}
        \sigma(x)=\sum_{j = 1}^\infty \varphi(x-j).
    \end{align}
    we have
    \begin{align}
        \sigma(x+1)-\sigma(x) &= \sum_{j = 1}^\infty 
        \varphi(x+1-j) - \sum_{j = 1}^\infty \varphi(x-j)\\
        &= \sum_{j = 0}^\infty 
        \varphi(x-j) - \sum_{j = 1}^\infty \varphi(x-j)\\
        &= \varphi(x).
    \end{align}
    Now, recall that we use the function $\varphi(x) =
    \frac{1}{\sqrt{2}} e^{-{x^2}/{2}}$. It is clear, 
    that the series 
    \begin{align}
        \sum_{j=1}^\infty |\varphi'(j-x)|
    \end{align}
    converges uniformly on any bounded interval. 
    This implies that $\sigma$ is differentiable and
    \begin{align}
        \sigma'(x) = \sum_{j = 1}^\infty \frac{(j-x)}{\sqrt{2}}e^{-{(x-j)^2}/{2}}.
    \end{align}
    Let us denote by $[x]$ the integer part of $x$ 
    and by $\{x\}= x - [x]$ the fractional part of $x$. 
    Recall also that the function $u\mapsto e^{-u^2/2}$
    is increasing on $(-\infty,0]$ and decreasing on
    $[0,+\infty)$. Therefore, we have
    \begin{align}
        \sigma(x) &= \frac{1}{\sqrt{2}}
        \sum_{j = 1}^{[x]}e^{-{(x-j)^2}/{2}} + 
        \frac{1}{\sqrt{2}}\sum_{j = [x]+1}^\infty e^{-{(x-j)^2}/{2}}\\
        &\le \frac{1}{\sqrt{2}}\sum_{j = 1}^{[x]-1} 
        \int_{x-j-1}^{x-j} e^{-u^2/{2}}\,du + 
        \frac{e^{-\{x\}^2/2} + e^{-(1-\{x\})^2/2}}{\sqrt{2}} 
        +\frac{1}{\sqrt{2}}\sum_{j = [x]+2}^\infty 
        \int_{j-x-1}^{j-x} e^{-u^2/{2}}\,du\\
        &\le \frac{1}{\sqrt{2}} 
        \int_{\{x\}}^{x-1} e^{-u^2/{2}}\,du + 
        \frac{e^{-\{x\}^2/2} + e^{-(1-\{x\})^2/2}}{\sqrt{2}} 
        +\frac{1}{\sqrt{2}} 
        \int_{-\infty}^{\{x\}-1} e^{-u^2/{2}}\,du\\
        &\le \frac{1}{\sqrt{2}} 
        \int_{-\infty}^{+\infty} e^{-u^2/{2}}\,du + 
        \frac1{\sqrt{2}} = \sqrt{\pi} + \frac1{\sqrt{2}}.
    \end{align}
    For $x>0$, using similar arguments and the fact that 
    the function
    $u\mapsto ue^{-u^2/2}$ is 
    decreasing on $[1,\infty)$, we get
    \begin{align}
        \sqrt{2}\sigma'(x) &= 
        -\sum_{j = 1}^{[x]} (x-j)e^{-{(x-j)^2}/{2}} + 
        \sum_{j = [x]+1}^\infty (j-x)e^{-{(j-x)^2}/{2}}\\
        &\le (1-\{x\}) e^{-(1-\{x\})^2/2} 
        +\sum_{j = [x]+2}^\infty 
        \int_{j-x-1}^{j-x} ue^{-u^2/{2}}\,du\\
        &\le (1-\{x\}) e^{-(1-\{x\})^2/2}  
        + \int_{1-\{x\}}^{\infty} ue^{-u^2/{2}}\,du\\
        & = (2-\{x\}) e^{-(1-\{x\})^2/2} \le \sqrt{2}.
    \end{align}
    In the same way, one can check that $\sqrt{2}\sigma'(x)\ge 
    -\sqrt{2}$ for every $x>0$. Therefore, $|\sigma'(x)|\le 1$
    for every positive $x$. On the other hand, for $x\le 0$, 
    we have $\sigma'(x)\ge 0$ and 
    \begin{align}
        \sigma'(x) &\le \sigma'(0) = \frac1{\sqrt{2}}
        \sum_{j=1}^\infty j e^{-j^2/2}\\
        &\le \frac1{\sqrt{2}}\bigg(e^{-1/2} + 
        \sum_{j=2}^\infty \int_{j-1}^{j} u e^{-u^2/2}\,du\bigg)\\
        & = \frac1{\sqrt{2}}\big(e^{-1/2} + e^{-1/2}\big) \le 1.
    \end{align}
    This completes the proof of the fact that $\sigma$ 
    is 1-Lipschitz.

\subsection{Proof of Proposition \ref{prop:ReLUSobolev}}\label{ap:ReLUsobolev}
\begin{proof}
Let assume  $r\in (\frac{D_0}{2}, 2D_0+2)$ and  $\bar r \in [D_0/2 , r)$. Then, $W^{r,2}(\mathcal{X}) \subset \mathscr{B}^{\bar r - D_0/2}(\mathcal{X})$  \cite[Lemma 2.5]{xu2020finite}, and since $\frac{2K}{2K+1}=\frac{2s+ D_0}{2s +2D_0+1}$, substituting  $s$ by $\bar r -D_0/2$, we obtain:

$$\frac{2K}{2K+1}=\frac{2\bar r}{2\bar r + D_0+1}. $$
Substituting the terms in \Cref{prop:ReLUBarron}, this yields the result with 
\begin{align}
    \bar g_{\bar r}(n) &= 2C_{\mathsf{PB}}  C^2+ 4C_{\mathsf{PB}}\log\left(3+
    \frac{3nB_1^2B_2^2(M_2^2 + \bar M_2^2)}{d\beta}\right).   
\end{align}
\end{proof}

\end{appendix}

\end{document}